\documentclass[11pt]{article}
\usepackage{amsfonts}
\usepackage{latexsym}
\usepackage{amsmath}
\usepackage{amssymb}
\usepackage{amsthm}
\usepackage{enumitem}
\usepackage{etoolbox}
\usepackage{ifthen}

\newcommand{\calI}{\mathcal{I}}
\newcommand{\expand}{\mathcal{E}}
\newcommand{\psdge}{\succcurlyeq}
\newcommand{\hess}[1]{{H_{J}}}
\newcommand{\eltwise}{\odot}
\newcommand{\normal}{\mathcal{N}}
\newcommand{\E}{\mathbb{E}}
\newcommand{\R}{\mathbb{R}}
\newcommand{\calC}{\mathcal{C}}

\newcommand{\pdiff}[2]{\frac{\partial #1}{\partial #2}}
\newcommand{\inr}[2]{\langle #1, #2 \rangle}
\newcommand{\grad}{\nabla}
\newcommand{\diffII}[3]{\ifthenelse{\equal{#2}{#3}}
{\frac{d^2 #1}{d #2^2}}
{\frac{d^2 #1}{d #2 d #3}}
}
\newcommand{\diff}[2]{\frac{d #1}{d #2}}
\newcommand{\pdiffII}[3]{\ifstrequal{#2}{#3}
{\frac{\partial^2 #1}{\partial #2^2}}
{\frac{\partial^2 #1}{\partial #2 \partial #3}}
}
\DeclareMathOperator{\diag}{diag}

\newtheorem{theorem}{Theorem}[section]
\newtheorem{lemma}[theorem]{Lemma}

\newtheorem{corollary}[theorem]{Corollary}

\theoremstyle{definition}

\theoremstyle{remark}

\newtheorem*{note*}{Note}



\begin{document}
\small

\title{\bf An interpolation proof of Ehrhard's inequality}

\author{Joe Neeman \ \ Grigoris  Paouris}

\maketitle

\begin{abstract}
  We prove Ehrhard's inequality using interpolation along the Ornstein-Uhlenbeck semi-group. 
  We also provide an improved Jensen inequality for Gaussian variables that might be
  of independent interest.
\end{abstract}

\bigskip

\medskip

\section{Introduction}

In \cite{Ehr}, A. Ehrhard proved the following Brunn-Minkowski like inequality for convex sets $A, B$ in $\mathbb R^{n}$: 
\begin{equation}
\label{Ehr}
\Phi^{-1} \left( \gamma_{n} ( \lambda A + (1-\lambda) B ) \right) \geq \lambda \Phi^{-1} (\gamma_{n} (A) ) + (1-\lambda) \Phi^{-1} (\gamma_{n} (B) ) , \ \lambda \in [0,1],
\end{equation}
where $ \gamma_{n}$ is the standard Gaussian measure in $\mathbb R^{n}$ (i.e. the measure with density $(2\pi )^{-n/2} e^{ - |x|^{2}/2} $) and $ \Phi$ is the Gaussian distribution function (i.e. $\Phi(x) = \gamma_{1} ( -\infty, x)$).  

This is a fundamental result of Gaussian space and it is known to have
numerous applications (see, e.g.,~\cite{L1}). Ehrhard's result was extended
by R. Lata\l{}a~\cite{L} to the case that one of the two sets is Borel and the other is
convex. Finally, C.\ Borell~\cite{Borell2} proved that it holds for all pairs of Borel sets.
Ehrhard's original proof for convex sets used a Gaussian symmetrization technique.
Borell used the heat semi-group and a maximal inequality in his proof,
which has since been further developed by Barthe and Huet~\cite{BartheHuet}; very
recently Ivanisvili and Volberg~\cite{IvanVol} developed this method into a general technique
for proving convolution inequalities. Another proof was recently found by van Handel~\cite{vanHandel}
using a stochastic variational principle.

In this work we will prove Ehrhard's inequality by constructing a quantity that is monotonic
along the Ornstein-Uhlenbeck semi-group. In
recent years this approach has been developed into a powerful tool to prove
Gaussian inequalities such as Gaussian hypercontractivity, the log-Sobolev
inequality, and isoperimetry~\cite{BGL}. There is no known
proof of Ehrhard inequality using these techniques and the purpose of
this note is to fill this gap.

An interpolation proof of the Lebesgue version of Ehrhard's inequality
(the Pr\'ekopa-Leindler inequality)
was presented recently in \cite{CDP}. This proof uses an
``improved reverse H\"older"
inequality for correlated Gaussian vectors that was established in
\cite{CDP}. A generalization of the aforementioned inequality also appeared
recently~\cite{Led,N1}. This inequality, while we call an ``improved
Jensen inequality'' for correlated Gaussian vectors, we present and actually
also extend in the present note. In \S2 we briefly discuss how this inequality
implies several known inequalities in probability, convexity and harmonic
analysis. Using a ``restricted'' version of this inequality (Theorem~\ref{thm:restricted-jensen}),
we will present a proof of Ehrhard's inequality. 

\smallskip

The paper is organized as follows: In \S2 we introduce the notation and basic
facts about the Ornstein-Uhlenbeck semi-group, and we present the proof of the
restricted, improved Jensen inequality. in \S3 we use Jensen inequality to
provide a new proof of Pr\'ekopa-Leindler inequality. We will use the main
ideas of this proof as a guideline for our proof of Ehrhard's inequality that
we present in \S4.

\section{An ``improved Jensen" inequality}

Fix a positive semi-definite $D \times D$ matrix $A$, and let $X \sim \normal(0, A)$.
For $t \ge 0$, we define the operator $P_t^A$ on $L_1(\R^{D}, \gamma_A)$ by
\[
  (P_t^A f)(x) = \E f(e^{-t} x + \sqrt{1-e^{-2t}} X).
\]
We will use the following well-known (and easily checked) facts:
\begin{itemize}
  \item the measure $\gamma_A$ is stationary for $P_t^A$;
  \item for any $s, t \ge 0$, $P_s^A P_t^A = P_{s + t}^A$;
  \item if $f$ is a continuous function having limits at infinity then
    $P_s^A f$ converges uniformly to $P_t^A f$ as $s \to t$.
\end{itemize}

We will heavily use the fact that $P_t^A$ commutes in a nice way with the composition of
smooth functions: let $\Psi: \R^k \to \R$ be a bounded $\calC^2$ function. For
any bounded, measurable $f = (f_1, \dots, f_k): \R^D \to \R^k$, any $x \in \R^D$
and any $0 < s < t$, $P_{t-s}^A \Psi(P_s^A f(x))$ is differentiable in $s$ and satisfies
\begin{equation}\label{eq:commutation}
  \pdiff{}{s} P_{t-s}^A \Psi(P_s^A f)
  = -P_{t-s}^A \sum_{i,j = 1}^k \partial_i \partial_j \Psi(f)
  \inr{\grad P_s^A f_i}{A \grad P_s^A f_j}.
\end{equation}

Suppose that $D = \sum_{i=1}^k d_i$, where $d_i \ge 1$ are integers.
We decompose $\R^D$ as $\prod_{i=1}^k \R^{d_i}$ and write
$\Pi_i$ for the projection on the $i$th component.
Given a $k \times k$ matrix $M$, write $\expand_{d_1, \dots, d_k} (M)$
for the $D \times D$ matrix whose $i,j$ entry is
$M_{k,\ell}$ if $\sum_{a < k} d_a < i \le \sum_{a \le k} d_a$
and $\sum_{b < \ell} d_b < j \le \sum_{b \le \ell} d_b$; that is, each entry
$M_{k,\ell}$ of $M$ is expanded into a $d_k \times d_\ell$ block.
We write `$\eltwise$' for the element-wise product of matrices,
`$\psdge$' for the positive semi-definite matrix ordering,
and $\hess J$ for the Hessian matrix of the function $J$.

Our starting point in this note is the following inequality,
which may be seen as an improved
Jensen inequality for correlated Gaussian variables.

\begin{theorem}\label{thm:jensen}
  Let $\Omega_1, \dots, \Omega_k$ be open intervals,
  and let $\Omega = \prod_{i=1}^k \Omega_i$.
  Take $X \sim \gamma_A$ and write $X_i = \Pi_i X$.
  For a bounded, $\calC^2$ function $J: \Omega \to \R$, the following are
  equivalent:
  \begin{enumerate}[label=(\ref{thm:jensen}.\alph*)]
    \item for every $x \in \Omega$, $A \eltwise \expand_{d_1, \dots, d_k}(\hess{J}(x)) \psdge 0$ \label{it:jensen-hess}
    \item for every $k$-tuple of measurable functions $f_i: \R^{d_i} \to \Omega_i$,
      \begin{equation}\label{eq:jensen}
        \E J(f_1(X_1), \dots, f_k(X_k))
          \ge J(\E f_1(X_1), \dots, \E f_k(X_k)).
        \end{equation}
      \label{it:jensen-ineq}
  \end{enumerate}
\end{theorem}

We remark that the restriction that $J$ be bounded can often be lifted.
For example, if $J$ is a continuous but unbounded function then one can still
apply Theorem~\ref{thm:jensen} on bounded domains $\Omega_i' \subset \Omega_i$.
If $J$ is sufficiently nice (e.g.\ monotonic, or bounded above) then one
can take a limit as $\Omega_i'$ exhausts $\Omega_i$ (e.g. using
the monotone convergence theorem, or Fatou's lemma).

As we have already mentioned, Theorem~\ref{thm:jensen} is known to have many consequences.
However, we do not know how to obtain Ehrhard's inequality using only Theorem~\ref{thm:jensen};
we will first need to extend Theorem~\ref{thm:jensen} in a few ways.
To motivate our first extension, note that the usual Jensen inequality on $\R$ extends easily to the
case where some function is convex only on a sub-level set. To be more precise, take
a function $\psi: \R \to \R$ and the set $B = \{x \in \R^d: \psi(x) < 0\}$. If
$B$ is connected and $\psi$ is convex when restricted to $B$, then one
easily finds that $\E \psi(X) \ge \psi(\E X)$ for any random vector
supported on $B$. A similar modification may be made to Theorem~\ref{thm:jensen}.

\begin{theorem}\label{thm:restricted-jensen}
  Take the notation and assumptions of Theorem~\ref{thm:jensen}, and assume in
  addition that $\{x \in \Omega: J(x) < 0\}$ is connected. Then the 
  following are equivalent:
  \begin{enumerate}[label=(\ref{thm:restricted-jensen}.\alph*)]
    \item for every $x \in \Omega$ such that $J(x) < 0$, $A \eltwise \expand_{d_1, \dots, d_k}(\hess{J}(x)) \psdge 0$
      \label{it:restricted-jensen-hess}
    \item for every $k$-tuple of measurable functions $f_i: \R^{d_i} \to \Omega_i$ that $\gamma_A$-a.s.\ satisfy
      $J(f_1, \dots, f_k) < 0$,
      \[
        \E J(f_1(X_1), \dots, f_k(X_k))
          \ge J(\E f_1(X_1), \dots, \E f_k(X_k)).
      \]
      \label{it:restricted-jensen-ineq}
  \end{enumerate}
\end{theorem}

Note that the threshold of zero in the conditions $J(x) < 0$
and $J(f_1, \dots, f_k) \le 0$ is arbitrary, since we may apply the
theorem to the function $J(\cdot) - a$ for any $a \in \R$. Of course,
taking $a$ sufficiently large recovers Theorem~\ref{thm:jensen}.

\begin{proof}
  Suppose that~\ref{it:restricted-jensen-hess} holds.
  By standard approximation arguments, it suffices to
  prove~\ref{it:restricted-jensen-ineq} for a more restricted class of
  functions $f$. Indeed, let $F$ be the set of measurable
  $f = (f_1, \dots, f_k)$ satisfying $J(f) < 0$ $\gamma_A$-a.s.\ and
  let $F_\epsilon \subset F$ be those functions that are
  continuous, vanish at infinity,
  and satisfy $J(f) \le -\epsilon$ $\gamma_A$-a.s. Now, every $f \in F$
  can be approximated pointwise by a sequence $f^{(n)} \in F_{1/n}$
  (here we are using the fact that $\{x: J(x) < 0\}$ is connected);
  hence, it suffices to prove~\ref{it:restricted-jensen-ineq} for
  $f \in F_\epsilon$, where $\epsilon > 0$ is arbitrarily small.
  From now on, fix $\epsilon > 0$ and fix $f = (f_1, \dots, f_k) \in
  F_\epsilon$.

  Recalling that $\Pi_i : \R^{d_1} \times \cdots \times \R^{d_k} \to \R^{d_i}$
  is the projection onto the $i$th block of coordinates, define
  $g_i = f_i \circ \Pi_i$ and $G_{s,t}(x) = P_{t-s}^A J(P_s^A g(x))$.
  Since $f \in F_\epsilon$, we have $G_{0,0}(x) \le -\epsilon$
  for every $x \in \R^D$. Moreover, since $f$ is continuous and
  vanishes at infinity, $P_s^A g \to g$ uniformly as $s \to 0$.
  Since $g$ is bounded, $J$ is uniformly continuous on the range of $g$
  and so there exists $\delta > 0$ such that
  $|G_{s,s}(x) - G_{r,r}(x)| < \epsilon$
  for every $x \in R^D$ and every $|s - r| \le \delta$.

  Now, fix $r \ge 0$ and assume that $G_{r,r} \le -\epsilon$ pointwise;
  by the previous paragraph, $G_{s,s} < 0$ pointwise for every
  $r \le s \le r + \delta$. Now we apply the commutation formula~\eqref{eq:commutation}:
  with $B_s = B_s(x) = A \eltwise \expand_{d_1, \dots, d_k}(H_J(P_s^A g))$, we have
  \[
    \pdiff{}{s} G_{s,t} = - P_{t-s}^A \sum_{{i,j}=1}^k
    \inr{\grad P_s^A g_i}{B \grad P_s^A g_j}
  \]
  (here, we have used the observation that $P_s^A g_i(x)$ depends only on $\Pi_i x$,
  and so $\grad P_s^A g_i$ is zero outside the $i$th block of coordinates).
  The assumption~\ref{it:restricted-jensen-hess} implies that $B_s$ is
  positive semi-definite whenever $G_{s,s} < 0$; since $G_{s,s} < 0$
  for every $s \in [r, r + \delta]$, we see that for such $s$,
  $\pdiff{}{s} G_{s,r + \delta} \le 0$ pointwise. Since
  $G_{s,r+\delta}$ is continuous in $s$ and $G_{r,r} \le -\epsilon$,
  it follows that $G_{s,s} \le -\epsilon$ pointwise for all
  $s \in [r, r+\delta]$.

  Next, note that $r=0$ satisfies the assumption $G_{r,r} \le -\epsilon$
  of the previous paragraph.
  By induction, it follows that $G_{r,r} \le -\epsilon$ pointwise for all
  $r \ge 0$. Hence, the matrix $B_s$ is positive semi-definite for all
  $s \ge 0$ and $x \in \R^D$, which implies that $G_{s,t}(x)$ is
  non-increasing in $s$ for all $t \ge s$ and $x \in \R^D$. Hence,
  \[
    \E J(f_1(X_1), \dots, f_k(X_k)) = \lim_{t \to \infty} G_{0,t}(0)
    \ge \lim_{t \to \infty} G_{t,t}(0) = J(\E f_1, \dots, \E f_k).
  \]
  This completes the proof of~\ref{it:restricted-jensen-ineq}.

  Now suppose that~\ref{it:restricted-jensen-ineq} holds.
  Choose some $v \in \R^D$ and some $y \in \Omega$ with $J(y) < 0$;
  to prove~\ref{it:restricted-jensen-hess}, it is enough to show that
  \begin{equation}\label{eq:restricted-jensen-part-2-goal}
    v^T (A \eltwise \expand_{d_1, \dots, d_k}(H_J(y))) v \ge 0.
  \end{equation}
  Since $\Omega$ is open, there is some $\delta > 0$ such that $y + z \in \Omega$
  whenever $\max_i |z_i| \le \delta$.
  For this $\delta$, define $\psi: \R \to \R$ by
  \[
    \psi(t) = \max\{-\delta, \min\{\delta, t\}\}.
  \]
  For $\epsilon > 0$, define $f_{i, \epsilon}: \R^{d_i} \to \Omega_i$
  by
  \[
    f_{i,\epsilon}(x) = y_i + \psi(\epsilon \inr{x}{\Pi_i v}).
  \]
  By~\ref{it:restricted-jensen-ineq},
  \[
    \E J(f_{1,\epsilon}(X_1), \dots, f_{k,\epsilon}(X_k))
    \ge J(\E f_{1,\epsilon}(X_1), \dots, \E f_{k,\epsilon}(X_k)).
  \]
  Since $\psi$ is odd, $\E f_{i,\epsilon}(X_i) = y_i$ for all $\epsilon > 0$; hence,
  \begin{equation}\label{eq:restricted-jensen-part-2-assumption}
    \E J(f_{1,\epsilon}(X_1), \dots, f_{k,\epsilon}(X_k)) \ge J(y).
  \end{equation}
  Taylor's theorem implies that for any $z$ with $y + z \in \Omega$,
  \[
    J(y + z)
    = J(y) + \sum_{i=1}^k \pdiff{J(y)}{y_i} z_i
    + \sum_{i,j=1}^k \pdiffII{J(y)}{y_i}{y_j} z_i z_j
    + \rho(|z|),
  \]
  where $\rho$ is some function satisfying $\epsilon^{-2} \rho(\epsilon) \to 0$ as $\epsilon \to 0$.
  Now consider what happens when we replace $z_i$ above with $Z_i = \psi(\epsilon\inr{X_i}{\Pi_i v})$ and
  take expectations. One easily checks that $\E Z_i = 0$, $\E \rho(|Z|) = o(\epsilon^2)$, and
  \[
    \E Z_i Z_j = \epsilon^2 (\Pi_i v)^T \E [X_i X_j] (\Pi_i v) + o(\epsilon^2);
  \]
  hence,
  \begin{align*}
    \E J(y + Z)
    &= J(y) + \epsilon^2 \sum_{i,j = 1}^k \pdiffII{J(y)}{y_i}{y_j} (\Pi_i v)^T \E [X_i X_j] (\Pi_i v) + o(\epsilon^2) \\
    &= J(y) + \epsilon^2 v^T (A \eltwise \expand_{d_1, \dots, d_k}(H_J(y))) v + o(\epsilon^2).
  \end{align*}
  On the other hand, $\E J(y + Z) = \E J(f_{1,\epsilon}(X_1), \dots, f_{k,\epsilon}(X_k))$,
  which is at least $J(y)$ according to~\eqref{eq:restricted-jensen-part-2-assumption}.
  Taking $\epsilon \to 0$ proves~\eqref{eq:restricted-jensen-part-2-goal}.
\end{proof}

\section{A short proof of Pr\'ekopa-Leindler inequality }

The Pr\'ekopa-Leindler inequality states that if $f, g, h: \R^d \to [0, \infty)$ satisfy
\[
  h(\lambda x + (1-\lambda) y) \ge f(x)^\lambda g(y)^{1-\lambda}
\]
for all $x, y \in \R^d$ and some $\lambda \in (0, 1)$ then
\[
  \E h \ge (\E f)^\lambda (\E g)^{1-\lambda},
\]
where expectations are taken with respect to the standard Gaussian measure on $\R^d$.
By applying a linear transformation, the standard Gaussian measure may be replaced by any
Gaussian measure; by taking a limit over Gaussian measures with large covariances, the expectations
may also be replaced by integrals with respect to the Lebesgue measure.

As M. Ledoux brought to our attention,
the Pr\'ekopa-Leindler inequality may be seen as a consequence of Theorem~\ref{thm:jensen};
we will present only the case $d=1$, but the case for general $d$ may be done in a similar way.
Alternatively, one may prove the Pr\'ekopa-Leindler inequality for $d=1$ first and then extend to
arbitrary $d$ using induction and Fubini's theorem.

Fix $\lambda \in (0, 1)$, let
$(X, Y) \sim \normal\big(0, \smash{\big(\begin{smallmatrix}1 & \rho \\ \rho & 1 \end{smallmatrix}\big)}\big)$
and let $Z = \lambda X + (1-\lambda) Y$. Let $\sigma^2 = \sigma^2(\rho, \lambda)$ be the variance
of $Z$ and let $A = A(\rho, \lambda)$ be the covariance of $(X, Y, Z)$.
Note that $A$ is a rank-two matrix, and that
it may be decomposed as $A = u u^T + v v^T$ where $u$ and $v$ are both orthogonal
to $(\lambda, 1-\lambda, -1)^T$.

For $\alpha, R \in \R_+$, define $J_{\alpha,R}: \R_+^3 \to \R$ by
\[
    J_{\alpha,R}(x, y, z) = (x^{\lambda} y^{1-\lambda} z^{-\alpha})^R.
\]

\begin{lemma}\label{lem:pl}
  For any $\lambda$ and $\rho$, and for any $\alpha < \sigma^2$,
  there exists $R \in \R_+$ such that $A \eltwise H_{J_{\alpha,R}} \psdge 0$.
\end{lemma}

To see how the Pr\'ekopa-Leindler inequality
follows from Theorem~\ref{thm:jensen} and Lemma~\ref{lem:pl}, suppose that
$h(\lambda x + (1-\lambda)y) \ge f^\lambda(x) + g^{1-\lambda}(y)$
for all $x, y \in \R$. Then $J_{\alpha,R}(f(X), g(Y), h^{1/\alpha}(Z)) \le 1$
with probability one (because $Z = \lambda X + (1-\lambda) Y$ with probability one).
By Theorem~\ref{thm:jensen}, with the $R$ from Lemma~\ref{lem:pl} we have
\begin{align*}
  1 &\ge \E J_{\alpha,R}(f(X), g(Y), h(Z)) \\
    &\ge J_{\alpha,R} (\E f(X), \E g(Y), \E h(Z)) \\
    &= \left(
      \frac{(\E f(X))^\lambda (\E g(Y))^{1-\lambda}}
      {(\E h^{1/\alpha}(Z))^\alpha}
    \right)^R.
\end{align*}
In other words, $(\E h^{1/\alpha}(Z))^\alpha \ge (\E f)^\lambda (\E g)^{1-\lambda}$.
This holds for any $\rho$ and any $\alpha < \sigma^2$. By sending $\rho \to 1$, we
send $\sigma^2 \to 1$ and so we may take $\alpha \to 1$ also. Finally, note
that in this limit $Z$ converges in distribution to $\normal(0, 1)$. Hence, we recover
the Pr\'ekopa-Leindler inequality for the standard Gaussian measure.

\begin{proof}[Proof of Lemma~\ref{lem:pl}]
  By a computation,
\begin{align*}
  H_{J_{\alpha,R}} &= J_{\alpha,R}(x, y, z)
    \begin{pmatrix}
      \frac{\lambda R(\lambda R - 1)}{x^2} & \frac{\lambda R(1-\lambda) R}{xy} & -\frac{\lambda \alpha R^2}{xz} \\
      \frac{\lambda R(1-\lambda) R}{xy} & \frac{(1-\lambda)R((1-\lambda)R - 1)}{y^2} & -\frac{(1-\lambda)\alpha R^2}{yz} \\
      -\frac{\lambda \alpha R^2}{xz} & -\frac{(1-\lambda)\alpha R^2}{yz} & \frac{\alpha R(\alpha R+1)}{z^2}
    \end{pmatrix}.
\end{align*}
We would like to show that $A \eltwise H_J \psdge 0$; since elementwise multiplication
commutes with multiplication by diagonal matrices, it is enough to show that
\begin{equation}\label{eq:goal}
    A \eltwise \left(
      \begin{pmatrix}
        \lambda \\
        1-\lambda \\
        -\alpha
      \end{pmatrix}^{\otimes 2}
      - \frac{1}{R} \begin{pmatrix}
        \lambda & 0 & 0 \\
        0 & 1-\lambda & 0 \\
        0 & 0 & -\alpha
      \end{pmatrix}
    \right) \ge 0.
\end{equation}
Let $\theta = (\lambda, 1-\lambda, -\alpha)^T$ and recall that $A = u u^T + v v^T$, where
$u$ and $v$ are both orthogonal to $(\lambda, 1-\lambda -1)^T$. Then
\[
    A \eltwise (\theta \theta^T) = (u \eltwise \theta) (u \eltwise \theta)^T +
    (v \eltwise \theta) (v \eltwise \theta)^T,
\]
where $u \eltwise \theta$ and $v \eltwise \theta$ are both orthogonal to $(1, 1, \frac{1}{\alpha})^T$
(call this $w$).
In particular, $A \eltwise (\theta \theta^T)$ is a rank-two, positive semi-definite matrix whose null
space is the span of $w$.

On the other hand,
$A \eltwise \diag(\lambda, 1-\lambda, -\alpha) = \diag(\lambda, 1-\lambda, -\alpha \sigma^2)$
(call this $D$). Then $w^T D w = 1 - \sigma^2/\alpha < 0$. As a consequence of the following Lemma,
\[
  A \circ (\theta \theta^T) - \frac 1R D \ge 0
\]
for all sufficiently large $R$.
\end{proof}

\begin{lemma}\label{lem:psd}
  Let $A$ be a positive semi-definite matrix and let $B$ be a symmetric matrix. If
  $u^T B u \ge \delta |u|^2$ for all $u \in \ker(A)$ and
  $v^T A v \ge \delta |v|^2$ for all $v \in \ker(A)^\perp$ then
  $A + \epsilon B \psdge 0$ for all $0 \le \epsilon \le \frac{\delta^2}{\|B\|^2 + \delta\|B\|}$,
  where $\|B\|$ is the operator norm of $B$.
\end{lemma}

\begin{proof}
  Any vector $w$ may be decomposed as $w = u + v$ with $u \in \ker(A)$ and $v \in \ker(A)^\perp$. Then
  \begin{align*}
    w^T (A + \epsilon B) w
    &= u^T A u + \epsilon u^T B u - 2 \epsilon u^T B v + \epsilon v^T B v \\
    &\ge \delta |u|^2 - \epsilon \|B\| |u|^2 - 2 \epsilon \|B\| |u| |v| + \epsilon \delta |v|^2.
  \end{align*}
  Considering the above expression as a quadratic polynomial in $|u|$ and $|v|$, we see that
  it is non-negative whenever $(\delta - \epsilon \|B\|) \delta \ge \epsilon \|B\|^2$.
\end{proof}

We remark that the preceding proof of the Pr\'ekopa-Leindler inequality may be extended in
an analogous way to prove Barthe's inequality~\cite{Ba1}.

\section{Proof of Ehrhard's inequality}

The parallels between the Pr\'ekopa-Leindler and Ehrhard inequalities become
obvious when they are both written in the following form.
The version of Pr\'ekopa-Leindler that we proved above may
be restated to say that
\begin{equation}\label{eq:pl}
\left.\begin{gathered}
  \exp(R (\lambda \log f(X) + (1-\lambda) \log g(Y) - \alpha \log h(Z))) \le 0
  \text{ a.s.} \\
  \text{implies} \\
  \exp(R (\lambda \log \E f(X) + (1-\lambda) \log \E g(Y) - \alpha \log \E h(Z))) \le 0.
\end{gathered}\right\}
\end{equation}
On the other hand, here we will prove that
\begin{equation}\label{eq:ehrhard}
\left.\begin{gathered}
  \Phi\left(R (\lambda \Phi^{-1}(f(X)) + (1-\lambda) \Phi^{-1}(g(Y))
  - \sigma \Phi^{-1}(h(Z)))\right) \le 0 \text{ a.s.} \\
  \text{implies} \\
  \Phi\left(R (\lambda \Phi^{-1}(\E f(X)) + (1-\lambda) \Phi^{-1}(\E g(Y))
  - \sigma \Phi^{-1}(\E h(Z)))\right) \le 0.
\end{gathered}\right\}
\end{equation}
(It may not yet be clear why the $\alpha$ in~\eqref{eq:pl} has become $\sigma$ in~\eqref{eq:ehrhard};
this turns out to be the right choice, as will become clear from the example in Section~\ref{sec:ehrhard-example}.)
This implies Ehrhard's inequality in the same way that~\eqref{eq:pl} implies
the Pr\'ekopa-Leindler inequality.
In particular, our proof of~\eqref{eq:pl} suggests
a strategy for attacking~\eqref{eq:ehrhard}: define the function
\[
J_R(x, y, z)
= \Phi\left(R (\lambda \Phi^{-1}(x) + (1-\lambda) \Phi^{-1}(y)
  - \sigma \Phi^{-1}(z))\right).
\]
(We will drop the parameter $R$ when it can be inferred
from the context.) In analogy with our proof of Pr\'ekopa-Leindler, we might then try to
show that for sufficiently large $R$, $A \eltwise H_{J_R} \psdge 0$. Unfortunately, this is false.

\subsection{An example}\label{sec:ehrhard-example}
Recall from the proof of Theorem~\ref{thm:restricted-jensen} that if
$A \eltwise H_J \psdge 0$ then
\[
G_{s, t, R}(x, y) :=
P_{t-s}^A J_R(P_s^1 f(x), P_s^1 g(y), P_s^{\sigma^2} h(\lambda x + (1-\lambda) y))
\]
is non-increasing in $s$ for every $x$ and $y$. We will give an example in which $G_{s,t,R}$ 
may be computed explicitly and it clearly fails to be non-increasing.

From now on, define $f_s = P_s^1 f$, $g_s = P_s^1 g$ and $h_s = P^{\sigma^2}_s h$.
Let $f(x) = 1_{\{x \le a\}}$,
$g(y) = 1_{\{y \le b\}}$ and $h(z) = 1_{\{z \le c\}}$, where
$c \ge \lambda a + (1-\lambda) b$. A direct computation yields
\begin{align*}
  f_s(x) &= \Phi\left(\frac{a - e^{-s} x}{\sqrt{1-e^{-2s}}}\right) \\
  g_s(y) &= \Phi\left(\frac{b - e^{-s} y}{\sqrt{1-e^{-2s}}}\right) \\
  h_s(z) &= \Phi\left(\frac{c - e^{-s} z}{\sigma \sqrt{1-e^{-2s}}}\right).
\end{align*}
Hence,
\[
J(f_s(x), g_s(y), h_s(\lambda x + (1-\lambda y)))
= \Phi\left(R\frac{\lambda a + (1-\lambda) b - c}{\sqrt{1-e^{-2s}}}\right).
\]
If $c > \lambda a + (1-\lambda)b$ then the above quantity is increasing in $s$.
Since it is also independent of $x$ and $y$, it remains unchanged when applying
$P_{t-s}^A$. That is,
\[
G_{s, t, R}
= \Phi\left(R\frac{\lambda a + (1-\lambda) b - c}{\sqrt{1-e^{-2s}}}\right)
\]
is increasing in $s$. On the bright side, in this example $G_{s,r,R\sqrt{1-e^{-2s}}}$
is constant. Since Theorem~\ref{thm:jensen} was not built to consider such behavior,
we will adapt it so that the function $J$
is allowed to depend on $s$.

\subsection{Allowing $J$ to depend on $t$}\label{sec:jensen-time-varying}

Recalling the notation of \S2, we assume from now on that $\Omega_i \subseteq [0, 1]$ for each $i$.
Then $A$ is a $k \times k$ matrix; let
$\sigma_1^2, \dots, \sigma_k^2$ be its diagonal elements.
We will consider functions of the form $J: \Omega \times [0, \infty] \to \R$. We write
$H_J$ for the Hessian matrix of $J$ with respect to the variables in $\Omega$, and $\pdiff J t$
for the partial derivative of $J$ with respect to the variable in $[0, \infty]$.
Let $I: [0, 1] \to \R$ be the function $I(x) = \phi(\Phi^{-1}(x))$.


\begin{lemma}\label{lem:time-varying-jensen}
  With the notation above, suppose that $J: \Omega \times [0, \infty] \to \R$ is bounded and $\calC^2$,
  and take $(X_1, \dots, X_k) \sim \gamma_A$.
  Let $\lambda_1, \dots, \lambda_k$ be non-negative numbers with $\sum_i \lambda_i = 1$, let $D(x)$
  be the $k \times k$ diagonal matrix with $\lambda_i \sigma_i^2 / I^2(x_i)$ in position $i$,
  and take some $\epsilon \ge 0$.
  If $\pdiff J t(x, t) \le 0$ and
  \begin{equation}
    A \eltwise \hess{J}(x, t) - (e^{2(t+\epsilon)} - 1) \pdiff{J(x,t)}{t} D^2 \psdge 0
    \label{eq:time-varying-hess}
  \end{equation}
  for every $x \in \Omega$ and $t > 0$ then
  for every $k$-tuple of measurable functions $f_i: \R \to \Omega_i$,
  \begin{equation}\label{eq:time-varying-jensen}
    \E J(P_\epsilon^{\sigma_1} f_1(X_1), \dots, P_\epsilon^{\sigma_k} f_k(X_k), 0)
    \ge J(\E f_1(X_1), \dots, \E f_k(X_k), \infty).
  \end{equation}
\end{lemma}

Note that Lemma~\ref{lem:time-varying-jensen} has an extra parameter $\epsilon \ge 0$ compared
to our previous versions of Jensen's inequality. This is for convenience when applying Lemma~\ref{lem:time-varying-jensen}:
when $\epsilon > 0$ then the function $e^{2(t + \epsilon)} - 1$
is bounded away from zero, which makes~\eqref{eq:time-varying-hess}
easier to check.

\begin{proof}
  Write $f_{i,s}$ for $P_{s+\epsilon}^{\sigma_i^2} f_i$ and $f_s = (f_{1,s}, \dots, f_{k,s})$.
  Define
  \[
    G_{s,t} = P_{t-s-\epsilon}^A J(f_{1,s}, \dots, f_{k,s}, s).
  \]
  We differentiate in $s$, using the commutation formula~\eqref{eq:commutation}.
  Compared to the proof of Theorem~\ref{thm:restricted-jensen},
  an extra term appears because the function $J$ itself depends on $s$:
  \begin{align*}
    - \pdiff{}{s} G_{s,t}
    &= P_{t-s-\epsilon} \sum_{i,j=1}^k \partial_i \partial_j J(f_s, s) A_{ij} f'_{i,s} f'_{j,s} - P_{t-s-\epsilon} \pdiff{J}{s}(f_s, s) \\
    &= P_{t-s-\epsilon} v_s^T (A \eltwise H_J(f_s,s)) v_s - P_{t-s-\epsilon} \pdiff{J}{s}(f_s,s),
  \end{align*}
  where $v_s = \grad f_s$.
  Bakry and Ledoux~\cite{BL} proved that
  $|v_{i,s}| \le \sigma_i^{-1} (e^{2(s+\epsilon)} - 1)^{-1/2} I(f_{i,s})$.
  Hence, 
  \[
    v_s^T D(f_s) v_s = \sum_{i=1}^k \lambda_i \left(\frac{\sigma_i |v'_{i,s}|}{I(f_{i,s})}\right)^2 \le (e^{2(s+\epsilon)} - 1)^{-1},
  \]
  and so
  \[
    - \pdiff{}{s} G_{s,t}
    \ge P_{t-s} \left(v_s^T (A \eltwise H_J(f_s,s)) v_s - (e^{2(s+\epsilon)} - 1) \pdiff{J}{s}(f_s,s) v_s^T D(f_s) v_s\right).
  \]
  Clearly, the argument of $P_{t-s}$ is non-negative pointwise if
  \[
    A \eltwise H_{J_R}(x,s) \psdge (e^{2(s+\epsilon)} - 1) \pdiff{J_R}{s}(x,s) D(x)
  \]
  for all $x, s$. In this case, $G_{s,t}$ is non-increasing in $s$ and we conclude as in
  the proof of Theorem~\ref{thm:restricted-jensen}.
\end{proof}

By combining the ideas of Theorem~\ref{thm:restricted-jensen} and Lemma~\ref{lem:time-varying-jensen},
we obtain the following combined version.

\begin{corollary}\label{cor:combined-jensen}
  With the notation and assumptions of Lemma~\ref{lem:time-varying-jensen}, suppose in addition
  that $\{x \in \Omega: J(x, 0) < a\}$ is connected, that $\pdiff{J(x,t)}{t} \le 0$ whenever $J(x, t) < a$,
  and that
  \[
    A \eltwise \hess{J}(x, t) - (e^{2(t+\epsilon)} - 1) \pdiff{J(x,t)}{t} D^2 \psdge 0
  \]
  for every $t \ge 0$ and every $x$ such that $J(x, t) < a$. Then for every $k$-tuple of measurable
  functions $f_i: \R \to \Omega_i$ satisfying $J(P_\epsilon^{\sigma_1} f_1, \dots, P_{\epsilon}^{\sigma_k} f_k) < 0$,
  \[
    \E J(P_\epsilon^{\sigma_1} f_1(X_1), \dots, P_\epsilon^{\sigma_k} f_k(X_k), 0) \ge J(\E f_1(X_1), \dots, \E f_k(X_k), \infty).
  \]
\end{corollary}

\subsection{The Hessian of $J$}
Define $J_R: (0, 1)^3 \to 0$ by
\[
  J_R(x, y, z)
  = \Phi\left(R \big(\lambda \Phi^{-1}(x) + (1-\lambda) \Phi^{-1}(y)
  - \sigma \Phi^{-1}(z))\big)\right).
\]
Let $H_J = H_J(x, y, z)$ denote the $3 \times 3$ Hessian matrix of $J$;
let $A$ be the $3 \times 3$ covariance
matrix of $(X, Y, Z)$. In order to apply Corollary~\ref{cor:combined-jensen},
we will compute the matrix $A \eltwise H_J$.
First, we define some abbreviations: set
\begin{align*}
  u &= \Phi^{-1}(x) & \Xi &= \lambda u + (1-\lambda) v - \sigma w \\
  v &= \Phi^{-1}(y) & \theta &= (\lambda, 1-\lambda, -\sigma)^T \\
  w &= \Phi^{-1}(z) & \calI &= \diag(\phi(u), \phi(v), \phi(w))
\end{align*}
We will use a subscript $s$ to denote that any of the above quantities is
evaluated at $(f_s, g_s, h_s)$ instead of $(x, y, z)$. That is
$u_s = \Phi^{-1}(f_s)$, $\Xi_s = \lambda u_s + (1-\lambda) v_s - \sigma w_s$,
and so on.

\begin{lemma}\label{lem:hess}
  $\displaystyle
      H_J = \phi(R\Xi) \calI^{-1} \left(
        R \diag(\lambda u, (1-\lambda)v, -\sigma w) - R^3 \Xi \theta \theta^T
      \right) \calI^{-1}.
  $
\end{lemma}

\begin{proof}
Noting that $\diff{u}{x} = 1/\phi(u)$,
the chain rule gives
\[
  \diff{}{x} \Phi(R \Xi) =
  R \lambda \frac{\phi(R \Xi)}{\phi(u)}
  = R \lambda \exp\left( -\frac{R^2 \Xi^2 - u^2}{2} \right).
\]
Differentiating again,
\[
  \diffII{}{x}{x} \Phi(R\Xi)
  = R \lambda (u - R^2 \Xi \lambda) \frac{\phi(R\Xi)}{\phi^2(u)}.
\]
For cross-derivatives,
\[
  \diffII{}{x}{y} \Phi(R\Xi)
  = - R^3 \Xi \lambda(1-\lambda) \frac{\phi(R\Xi)}{\phi(u)\phi(v)}.
\]
Putting these together with the analogous terms involving differentiation
by $z$,
\begin{multline*}
  \frac{H_J}{\phi(R\Xi)} =
  - R^3\Xi \begin{pmatrix}
    \frac{\lambda^2}{\phi^2(u)} & \frac{\lambda(1-\lambda)}{\phi(u)\phi(v)} & -\frac{\lambda \sigma}{\phi(u)\phi(w)} \\
    \frac{\lambda(1-\lambda)}{\phi(u)\phi(v)} & \frac{(1-\lambda)^2}{\phi^2(v)} & -\frac{(1-\lambda)\sigma}{\phi(v)\phi(w)} \\
    -\frac{\lambda\sigma}{\phi(u)\phi(w)} & -\frac{(1-\lambda)\sigma}{\phi(u)\phi(v)} & \frac{\sigma^2}{\phi^2(w)}
  \end{pmatrix} \\
  + R \begin{pmatrix}
    \frac{\lambda u}{\phi^2(u)} & 0 & 0 \\
    0 & \frac{(1-\lambda) v}{\phi^2(v)} & 0 \\
    0 & 0 & -\frac{\sigma w}{\phi^2(w)}
  \end{pmatrix}.
\end{multline*}
Recalling the definition of $\calI$ and $\theta$, this may
be rearranged into the claimed form.
\end{proof}

Having computed $H_J$, we need to examine $A \odot H_J$. Recall that $A$
is a rank-two matrix and so it may be decomposed as $A = a a^T + b b^T$. Moreover,
the fact that $Z = \lambda X + (1-\lambda) Y$ means that $a$ and $b$
are both orthogonal to $(\lambda, 1-\lambda, -1)^T$. Recalling the definition
of $\theta$, this implies that $a \eltwise \theta$ and $b \eltwise \theta$
are both orthogonal to $(1, 1, \sigma^{-1})^T$.
This observation allows us to deal with the $\theta \theta^T$ term in
Lemma~\ref{lem:hess}:
\[
A \eltwise \theta \theta^T = (a a^T) \eltwise (\theta \theta^T) + (b b^T) \eltwise (\theta \theta^T)
= (a \eltwise \theta)^{\otimes 2} + (b \eltwise \theta)^{\otimes 2}.
\]
To summarize:

\begin{lemma}
  The matrix $B := A \eltwise \theta \theta^T$ is positive semidefinite and has rank two.
  Its kernel is the span of $(1, 1, \frac 1\sigma)^T$.
\end{lemma}

On the other hand, the diagonal entries of $A$ are $1, 1,$ and $\sigma^2$; hence,
\[
A \eltwise \diag(\lambda u, (1-\lambda)v, -\sigma w)
= \diag(\lambda u, (1-\lambda)v, - \sigma^3 w) =: D.
\]
Combining this with Lemma~\ref{lem:hess}, we have
\begin{equation}\label{eq:A-hess}
A \eltwise H_J = R \phi(R\Xi) \calI^{-1}
(D - R^2 \Xi B)
\calI^{-1}.
\end{equation}

Consider the expression above in the light of our earlier proof of
Pr\'ekopa-Leindler.  Again, we have a sum of two matrices ($D$ and $-R^2 \Xi B$),
one of which is multiplied by a factor ($R^2$) that we may take to be large.
There are two important differences.
The first is that the matrix $D$ (whose analogue was constant in the proof of Pr\'ekopa-Leindler) cannot be
controlled pointwise in terms of $B$. This difference is closely related to
the example in Section~\ref{sec:ehrhard-example}; we will solve it by making $J$
depend on $t$ in the right way; the $\diff{J}{t}$ term in Corollary~\ref{cor:combined-jensen}
will then cancel out part of $D$'s contribution.

The second difference is
that in~\eqref{eq:A-hess}, the term that is multiplied by a large factor
(namely, $-\Xi B$) is not everywhere positive semi-definite because there exist
$(x, y, z) \in \R^3$ such that $\Xi(x, y, z) > 0$. This is the reason that we consider
the ``restricted'' formulation of Jensen's inequality in Theorem~\ref{thm:restricted-jensen}
and Corollary~\ref{cor:combined-jensen}.

\subsection{Adding the dependence on $t$}
Recall that $X$ and $Y$ have variance 1 and covariance $\rho$, that $Z = \lambda X + (1-\lambda) Y$,
and that $A$ is the covariance of $(X, Y, Z)$.
Recall also the notations $u, v, w, \Xi$, and their subscripted variants.
For $R > 0$, define $r(t) = R \sqrt{1-e^{-2t - \epsilon}}$ and
\begin{align}
  J_R(x, y, z, t)
  &= \Phi\left(r(t) \big(\lambda \Phi^{-1}(x) + (1-\lambda) \Phi^{-1}(y)
  - \sigma \Phi^{-1}(z))\big)\right) \notag \\
  &= \Phi(r(t) \Xi). \label{eq:J-def}
\end{align}
Let $E = \diag(\lambda, 1-\lambda, \sigma) / (1 + \sigma^{-1})$.

\begin{lemma}\label{lem:combined-jensen-condition}
  Define $\Omega_\epsilon = [\Phi(-1/\epsilon), \Phi(1/\epsilon)]^3$.
  For every $\rho, \lambda$, and $\epsilon$, there exists $R > 0$ such that
  \[
    A \eltwise H_J - (e^{2(t+\epsilon)} - 1) \pdiff{J}{t} \calI^{-1} E \calI^{-1} \psdge 0
  \]
  on $\{(x, t) \in \Omega_\epsilon \times [0, \infty): \Xi(x) \le -\epsilon\}$.
\end{lemma}

\begin{proof}
  We computed $A \eltwise H_J$ in~\eqref{eq:A-hess} already; applying that formula
  and noting that $\calI^{-1} \psdge 0$, it suffices to show that
  \[
    r(t) \phi(r(t) \Xi) (D - r^2(t) \Xi B) - (e^{2(t+\epsilon)} - 1) \pdiff{J}{t} E \psdge 0
  \]
  whenever $\Xi \le -\epsilon$.
  (Recall that $D = \diag(\lambda u, (1-\lambda) v, -\sigma^3 w)$, and that $B$ is a rank-two
  positive semidefinite matrix that depends only on $\rho$ and $\lambda$, and whose kernel is
  the span of $(1, 1, \sigma^{-1})^T$). We compute
  \[
    \pdiff{J}{t} = r'(t) \Xi \phi(r(t) \Xi) = \frac{r(t)}{e^{2t+\epsilon} - 1} \Xi \phi(r(t) \Xi).
  \]
  Now, there is some $\delta = \delta(\epsilon) > 0$ such that
  \[
    \frac{e^{2(t+\epsilon)} - 1}{e^{2t + \epsilon} - 1} \ge 1 + \delta
  \]
  for all $t \ge 0$. For this $\delta$,
  \begin{multline*}
    r(t) \phi(r(t) \Xi) (D - r^2(t) \Xi B) - (e^{2(t+\epsilon)} - 1) \pdiff{J}{t} E \\
    \psdge
    r(t) \phi(r(t) \Xi) (D - (1+\delta) \Xi E - r^2(t) \Xi B);
  \end{multline*}
  Hence, it suffices to show that $D - (1 + \delta) \Xi E - r^2(t) \Xi B \psdge 0$.
  Since $\Xi \le -\epsilon$, it suffices to show that
  $r^2(t) \epsilon B + D - (1 + \delta) \Xi E \psdge 0$.
  Now, $B$ is a rank-two positive semi-definite matrix depending only on $\lambda$ and $\rho$.
  Its kernel is spanned by $\theta = (1, 1, \sigma^{-1})^T$. Note that $\theta^T D \theta = \Xi$
  and $\theta^T E \theta = 1$. Hence,
  \[
    \theta^T (D - (1 + \delta) \Xi E) \theta = - \delta \Xi \ge \delta \epsilon > 0.
  \]

  Next, note that we can bound the norm of $D - (1 + \delta) \Xi E$ uniformly:
  on $\Omega_\epsilon$, $\|D\| \le 1/\epsilon$
  and $|\Xi| \le 2/\epsilon$. All together, if we assume (as we may) that $\delta \le 1$ then
  $\|D + (1 + \delta) \Xi E\| \le 5 / \epsilon$. By Lemma~\ref{lem:psd}, if $\eta > 0$
  is sufficiently small then
  \[
    \epsilon B + \eta(D - (1 + \delta) \Xi E) \psdge 0.
  \]
  To complete the proof, choose $R$ large enough so that $R^2 (1-e^{\epsilon}) \ge 1/\eta$;
  then $r^2(t) \ge 1/\eta$ for all $t$.
\end{proof}

Finally, we complete the proof of~\eqref{eq:ehrhard} by a series of simple approximations.
First, let $C_a$ denote the set of continuous functions $\R \to [0, 1]$ that converge to $a$ at $\pm \infty$,
and note that it suffices to prove~\eqref{eq:ehrhard} in the case that $f, g \in C_0$ and $g \in C_1$. Indeed,
any measurable $f, g: \R \to [0, 1]$ may be approximated (pointwise at $\gamma_1$-almost every point)
from below by functions in $C_0$, and any measurable $h: \R \to [0, 1]$ may be approximated from above by functions in $C_1$.
If we can prove~\eqref{eq:ehrhard} for these approximations, then it follows (by the dominated convergence theorem)
for the original $f, g$, and $h$.

Now consider $f, g \in C_0$ and $h \in C_1$ satisfying $\Xi(f, g, h) \le 0$ pointwise. For $\delta > 0$, define
\begin{align*}
  f_\delta &= \Phi(-1/\delta) \lor f \land \Phi(1/(3\delta)) \\
  g_\delta &= \Phi(-1/\delta) \lor g \land \Phi(1/(3\delta)) \\
  h_\delta &= \Phi \left( - \frac{1}{3\delta} \lor (\Phi^{-1}(h) + \delta) \land \frac{1}{\delta} \right).
\end{align*}
If $\delta > 0$ is sufficiently small then $\Xi(f_\delta, g_\delta, h_\delta) \le -\delta$ pointwise;
moreover, $f_\delta, g_\delta$, and $h_\delta$ all take values in $[\Phi(-1/\delta), \Phi(1/\delta)]$, are
continuous, and have limits at $\pm \infty$. Since $f_\delta \to f$ as $\delta \to 0$ (and similarly for $g$ and $h$),
it suffices to show that
\begin{equation}\label{eq:fdelta}
  \lambda \Phi^{-1}(\E f_\delta) + (1-\lambda) \Phi^{-1}(\E g_\delta) \le \sigma \Phi^{-1}(\E h_\delta)
\end{equation}
for all sufficiently small $\delta > 0$.

Since $f_\delta$ has limits at $\pm \infty$,
it follows that $P_\epsilon f_\delta \to f_\delta$ uniformly
as $\epsilon \to 0$ (similarly for $g_\delta$ and $h_\delta$). By taking $\epsilon$ small enough (at least as
small as $\delta/2$), we can ensure
that $\Xi (P^1_\epsilon f_\delta, P^1_\epsilon g_\delta, P^{\sigma^2}_\epsilon h_\delta) < -\epsilon$ pointwise.
Now we apply Corollary~\ref{cor:combined-jensen} with $\Omega_i = [\Phi(-1/\epsilon), \Phi(1/\epsilon)]$,
the function $J$ defined in~\eqref{eq:J-def}, $a = \frac 12$, and
with $(\lambda_1, \lambda_2, \lambda_3) = (\lambda, 1-\lambda, \sigma^{-1})/(1 + \sigma^{-1})$.
Lemma~\ref{lem:combined-jensen-condition} implies that the condition of Corollary~\ref{cor:combined-jensen}
is satisfied. We conclude that
\begin{align*}
  \frac 12
  &\ge J_R(\E f_\delta, \E g_\delta, \E h_\delta, \infty) \\
  &= \Phi\left(
    R\big(\lambda \Phi^{-1}(\E f_\delta) + (1-\lambda) \Phi^{-1} (\E g_\delta) - \sigma \Phi^{-1} (\E h_\delta)\big)
  \right),
\end{align*}
which implies~\eqref{eq:fdelta} and completes the proof of~\eqref{eq:ehrhard}.

\section*{Acknowledgements}

We thank F. Barthe and M. Ledoux for helpful comments and for directing them to related literature.

We would also like to thank R. van Handel for pointing out to us that~\eqref{eq:ehrhard} corresponds more
directly to a generalized form of Ehrhard's inequality contained in Theorem 1.2 of~\cite{Borell3}.

\bigskip

\bigskip


\footnotesize
\bibliographystyle{amsplain}

\begin{thebibliography}{100}

\bibitem{BL} {\rm D. Bakry and M. Ledoux}, {\sl Levy-Gromov's isoperimetric inequality for an infinite dimensional diffusion generator}, Invent. Math. 123, (1996), no. 1, 259--281.

\bibitem{Ba1} {\rm  F. Barthe},  {\sl On a reverse form of the Brascamp-Lieb inequality}, Invent. Math. 134 (1998), 335--361.

\bibitem{BartheHuet} {\rm F. Barthe and N. Huet}, {\sl On Gaussian Brunn-Minkowski inequalities}, Studia Math. 191 (2009), no. 3, 283--304.

\bibitem{BGL} {\rm D. Bakry, I. Gentil and M. Ledoux}, {\sl Analysis and geometry of Markov diffusion operators}. Springer, 2013.

\bibitem{Borell1} {\rm C. Borell}, {\sl The Brunn-Minkowski inequality in Gauss space}, Invent. Math. 30, 2 (1975), 207--216. 

\bibitem{Borell2} {\rm C. Borell}, {\sl The Ehrhard inequality}, C.R. Math. Acad. Sci. Paris 337 (2003), no. 10, 663--666. 

\bibitem{Borell3} {\rm C. Borell}, {\sl Minkowski sums and Brownian exit times}, Ann. Fac. sci. Toulouse: Math., 16, no. 1 (2007), pp. 37--47.

\bibitem{CLL} {\rm  E. A. Carlen, E. H. Lieb and M. Loss},  {\sl A sharp analog of Young's inequality on $S^{N}$ and related entropy inequalities}, J. Geom. Anal., 14, no. 3, pp. 487--520.

\bibitem{CDP} {\rm W-K. Chen, N. Dafnis and G. Paouris}, {\sl Improved Holder and reverse Holder inequalities for Gaussian random vectors}, Adv. Math. 280 (2015), 643--689. 

\bibitem{Ehr} {\rm A. Ehrhard}, {\sl Sym\'etrisation dans l'espace de Gauss}, Math. Scand. 53 (1983) 281--301.

\bibitem{Ga} {\rm  R. Gardner}, {\sl The Brunn-Minkowski inequality}, Bull. Amer. Math. Soc. 39, (2002) pp. 355--405.

\bibitem{IvanVol} {\rm P. Ivanisvili and A. Volberg}, {\sl Bellman partial differential equation and the hill property for classical isoperimetric problems}, ArXiv preprint {\tt arXiv:1506.03409}.

\bibitem{L} {\rm R. Latala}, {\sl A note on the Ehrhard inequality}, Studia Math. 118 (1996), 169--174.  

\bibitem{L1} {\rm R. Latala}, {\sl On some inequalities for Gaussian measures}, Proc. ICM (2002), vol. II, 813--822.

\bibitem{Led} {\rm M. Ledoux}, {\sl Remarks on Gaussian noise stability, Brascamp-Lieb and Slepian inequalities}, Geom. Aspects Func. Anal. (2014), 309--333.

\bibitem{N1} {\rm J. Neeman}, {\sl A multi-dimensional version of noise stability}, Electron. Commun. Probab. 19 (2014), no. 72, 1--10.

\bibitem{vanHandel} {\rm R. van Handel}, {\sl The Borell-Ehrhard Game}, ArXiv preprint {\tt arXiv:1605.00285}.

\end{thebibliography}

\bigskip

\noindent \textsc{Joe Neeman}

\noindent \textsc{Institute of Applied Mathematics}

\noindent \textsc{University of Bonn}

\noindent \textsc{Bonn, 53113 Germany}

\noindent \textit{E-mail:} {\tt joeneeman@gmail.com}

\bigskip

\noindent \textsc{Grigoris Paouris}

\noindent \textsc{Department of Mathematics}

\noindent \textsc{Texas A \& M University}

\noindent \textsc{College Station, TX 77843 U.S.A.}

\noindent \textit{E-mail:} {\tt grigorios.paouris@gmail.com}, {\tt grigoris@math.tamu.edu}

\vspace*{1in}

\small
\end{document}